\documentclass[a4paper,12pt]{article}

\usepackage{amsthm,amssymb}
\usepackage{graphicx,color}
\usepackage{subfigure}
\usepackage{epstopdf}
\usepackage{latexsym}
\usepackage{amsmath}
\usepackage{tikz}
\def\p{\partial}
\def\na{\nabla}

\def\x{{\mathbf x}}

\def\f{\frac}
\def\q{\quad}
\def\Om{\Omega}
\def\om{\omega}
\def\ep{\epsilon}

\def\mC{{\mathcal C}}
\def\mE{{\mathcal E}}
\def\Q{\mathbf Q}
\def\P{\mathbf P}
\def\a{\mathbf a}
\def\z{\mathbf z}
\newtheorem{theorem}{Theorem}[section]
\newtheorem{lemma}[theorem]{Lemma}

\newtheorem{remark}[theorem]{Remark}

\numberwithin{equation}{section}

\title{Electrical impedance spectroscopy-based nondestructive testing for imaging defects in concrete structures}
\author{Habib Ammari\footnotemark[2] \and  Jin Keun Seo\footnotemark[3] \and Tingting Zhang\footnotemark[3]  \and Liangdong Zhou\footnotemark[3]}
\begin{document}
\maketitle
\renewcommand{\thefootnote}{\fnsymbol{footnote}}
\footnotetext[2]{Department of Mathematics and Applications, Ecole
Normale Sup\'erieure, 45 Rue d'Ulm, 75005 Paris, France ({\tt
habib.ammari@ens.fr}).} \footnotetext[3]{Department of
Computational Science and Engineering, Yonsei University, Seoul 120-749
Korea, ({\tt seoj@yonsei.ac.kr, zttouc@hotmail.com,
zhould1990@hotmail.com}).}
\renewcommand{\thefootnote}{\arabic{footnote}}
\begin{abstract}
An electrical impedance spectroscopy-based nondestructive testing
(NDT) method is proposed to image both cracks and reinforcing bars
in concrete structures. The method utilizes the
frequency-dependent behavior of thin insulating cracks:
low-frequency electrical currents are blocked by insulating
cracks, whereas high-frequency currents can pass through the
conducting bars without being blocked by thin cracks. Rigorous
mathematical analysis relates the geometric structures of the
cracks and bars to the frequency-dependent Neumann-to-Dirichlet
data. Various numerical simulations support the feasibility of the
proposed method.
\end{abstract}

\thispagestyle{plain}
\markboth{H. Ammari, J. K. Seo, T. Zhang, and L. Zhou}{EIS-based NDT for imaging cracks}
\section{Introduction}
As a number of concrete structures currently in service reach the
end of their expected serviceable life, nondestructive testing
(NDT) methods to evaluate their durability, and thus to ensure
their structural integrity, have received gradually increasing
attention. Concrete often degrades by the corrosion of the
embedded reinforcing bars, which can lead to internal stress and
thus to structurally disruptive cracks \cite{Feldmann2008}.
Various NDT techniques are currently used to monitor the
reliability and condition of reinforced concrete structures
without causing damage. They include impact-echo, half-cell
potential, electrical resistivity testing, ground penetrating
radar, ultrasonic testing, infrared thermographic techniques, and
related tomographic imaging techniques
\cite{Colla1995,Davis1989,Guide,McCann2001,Sansalone1988,Sansalone1986,Yin2010}.
Each technique has its intrinsic limitations in terms of
reliability of defect detection, and the conventional techniques
often depend on the subjective judgment of the inspectors. The
limitations of existing methods have led to searches for more
advanced visual inspection methods to detect invisible flaws and
defects on the surface of concrete structures.

Electric methods such as electrical resistivity tomography (ERT)
and electrical capacitance tomography (ECT) have been used to
image cracks and steel reinforcing bars, which show clear
electrical contrast from the background concrete. These electric
methods can be used to complement acoustic methods by assessing
different characteristics. They operate at low cost over long time
periods. ERT and ECT employ multiple current sources to inject
currents, and boundary voltages are then measured using voltmeters
connected to multiple surface electrodes on the boundary of the
imaging subject. These methods use the relationship between the
applied current and the measured boundary voltage to invert the
image of cracks and reinforcing bars. The methods suffer from a
low defect location accuracy due to the ill-posedness of the
corresponding inverse problem. In fact, the boundary
current-voltage measurement alone may not be sufficient for robust
identification of defects. Most of research outcomes had
cooperated  with some form of a prior information to deal with
ill-posedness \cite{Ammari2009a, Ammari2009b, Ammari2011,
Ammari2008, Buettner1996, Beau2002,
Diamond2006,Hou2006,Kaipio2010a, Kaipio2009, Kaipio2010b}. They
suffer from low defect location accuracy owing to the ill-posed
nature of the corresponding inverse problem. The boundary
current-voltage measurement alone may not be sufficient for the
robust identification of defects. Most of the previous methods can
detect cracks and reinforcing bars when either only the crack or
the bar exists in the concrete samples. Numerous experiments show
that ERT or ECT applied with a single frequency struggle to
identify both cracks and reinforcing bars when electrical currents
are blocked by insulating cracks near the electrodes.

This paper focuses on electrical impedance spectroscopy-based NDT, viewing as an integrated ERT/ECT modality. It provides mathematical analysis to support a better method of visually inspecting defects in concrete structures.   A thorough understanding of the frequency-dependent effects of thin insulating cracks could be used to image cracks and reinforcing bars: low-frequency electrical currents are blocked by insulating cracks, whereas high-frequency currents can pass through the conducting bars without being blocked by thin cracks.

The mathematical analysis assumes that the background concrete is
roughly homogeneous. The effective admittivity of the concrete
could be regarded as roughly constant at coarse grid, despite it
comprising a complex mixture of several materials. The proposed
impedance-spectroscopy-based NDT method can provide visual
assessment of the condition of a concrete structure, instead of
coarse structure information, through tomographic images of the
effective admittivity of the heterogeneous concrete structure.

In this paper, the cracks are modeled as thin inhomogeneities in
the concrete, while the reinforcing bars are modeled as small
inhomogeneities \cite{Ammari2006, Ammari2004, Ammari2007,
Beretta2003, Beretta2001, Frieman1989}. Two operating frequency
regimes are considered: low and high. Based on  \cite{Zribi2011},
the corresponding asymptotic expansion of the boundary voltage is
established here at these two frequency regimes.

The main purpose is to show that multi-frequency impedance
measurements can be used to visualize different objects. For
simplicity, cracks are idealized as linear segments. The numerical
simulations use a conventional 16-channel EIT system, with the
electrical current applied between two adjacent electrodes at
different frequencies. The boundary voltage data are then measured
between two adjacent electrodes attached on the surface.
Frequency-difference EIT reconstruction allows the detection of
both the cracks and the reinforcing bars within the concrete
structures. A variety of numerical experiments is presented here
to illustrate the main findings.

\section{Mathematical model}

For  rigorous analysis, we use the simplified two-dimensional
model by considering axially symmetric cylindrical sections under
the assumption that the out-of-plane current density is negligible
in an imaging slice. We assume a two dimensional electrically
conducting domain $\Om$  with its connected $C^2$-boundary $\Om$.
We denote the conductivity distribution of the domain by $\sigma$
and the permittivity distribution by $\epsilon$. Inside $\Om$,
there exist thin cracks $\mC_k,~k=1,2,\cdots, N_C$ and reinforcing
bars $D_k,~ k=1,2,\cdots, N_D$ as shown in figure
\ref{Fig:Cmodel}.   Let $D=\cup_{k=1}^{N_D} D_k$ and
$\mC=\cup_{k=1}^{N_C} \mC_k$ denote the collections of the
reinforcing bars and  cracks, respectively.  Since the
conductivity $\sigma$ and permittivity $\ep$ change abruptly
across the reinforcing bars and cracks, we denote
\begin{equation}\label{Eq:Admt}
\sigma(x)=
\left\{\begin{array}{ll}
\sigma_c  &\q \mbox{for}~x\in  \mathcal{C},\\
 \sigma_d   &\q \mbox{for}~x\in D, \\
  \sigma_b  &\q\mbox{for}~x\in\Omega\backslash(D \cup \mathcal{C}), \\
\end{array}
\right. \mbox{ and } \epsilon(x)=
\left\{\begin{array}{ll}
 \epsilon_c  &\q \mbox{for}~x\in  \mathcal{C},\\
 \epsilon_d  &\q \mbox{for}~x\in D, \\
 \epsilon_b &\q\mbox{for}~x\in\Omega\backslash(D \cup \mathcal{C}). \\
\end{array}
\right.
\end{equation}
Because the cracks are highly insulating and the reinforcing bars are highly conducting, we consider the following two extreme contrast cases:
$$
\f{\sigma_c}{\sigma_b}\approx 0\q\mbox{and }\q \f{\sigma_d}{\sigma_b}\approx \infty.
$$

The inverse problem is to identify the cracks $\mC_k$ and
reinforcing bars $D_k$ from measured current-voltage data in
multi-frequency EIT system. In the frequency range below 1MHz
($\f{\om}{2\pi}\le 10^6$), we inject a sinusoidal current
$g(x)\sin (\om t)$ at $x\in\p\Om$ where $g$ is the magnitude of
the current density on $\p\Om$ and $g\in
H^{-1/2}_\diamond(\p\Om):=\{ \phi\in H^{-1/2}(\p\Om)~:~
\int_{\p\Om} \phi ds=0\}$. The injected current produces the
time-harmonic potential $u^\om$ in $\Om$ which is dictated by
\begin{equation}\label{Eq:uw}
\left\{
\begin{array}{ll}
 \nabla\cdot\left(\gamma^\omega(x)\nabla u^\omega(x)\right) =0 &\quad\mbox{in}~\Omega,\\
 \gamma^\omega \f{\partial u^ \omega}{\p \nu} = g &\quad\mbox{on}~\partial\Omega,
 \end{array}\right.
\end{equation}
where $ \gamma^{\om} = \sigma + i\omega \ep$, ${\nu}$ is the
outward unit normal vector on $\p\Om$, and $\f{\p}{\p\nu}$ is the
normal derivative. Setting $\int_{\p\Om} u^\om ds=0$, we can
obtain a unique solution $u^\om$ to (\ref{Eq:uw}) from the
Lax-Milgram theorem. Hence, we can define the Neumann to Dirichlet
map $\Lambda_\omega:H^{-1/2}_\diamond(\p\Omega) \to
H^{1/2}_\diamond(\p\Omega)$ by
$\Lambda_\omega(g)=u^\om|_{\p\Omega}$. Using $N_E-$channel
multi-frequency EIT system, we may inject $N_E$ number of linearly
independent currents at several angular frequencies
$\om_1,\cdots,\om_{N_\om}$ and measure the induced corresponding
boundary voltages. We collect these current-voltage data
$\{\Lambda_{\om_j}(g_k)~:~~k=1,\cdots, N_E,~j=1,\cdots,
N_\omega\}$ at various frequencies ranging from 10Hz to 1MHz which
will be used to detect cracks and reinforcing bars.

To carry out rigorous analysis, we will restrict our considerations to geometric structures of $\mC$ and $D$ as shown in Figure \ref{Fig:Cmodel}.  We assume that each crack $\mC_k$ has a uniform thickness of $\delta_k$ and is a neighborhood of a $C^2-$smooth open curve $\mathcal L_k$:
\begin{equation}\label{Eq:curve2}
  \mathcal{C}_k = \{x+h \nu_x~:~ x\in \mathcal L_k, ~-\delta_k<h<\delta_k  \},~~~(k=1,2,\cdots,N_C).
\end{equation}
The thickness to the crack length ratio is assumed to be very small, that is,  $\delta_k\approx 0$. We also assume that each reinforcing bar has the form
\begin{equation}\label{Dk}
D_k :=z_k + \delta_D B_k,~~~ (k=1,2,\cdots,N_D),
\end{equation}
where $B_k$ is a bounded smooth reference domain centered at $(0,0)$ and $\delta_D$ is related to the diameter of $D_k$.
\begin{figure}[ht!]
\begin{center}
\begin{tikzpicture}[scale=1.2]

  \draw[draw=gray,fill=white] [line width=1pt](0,0) circle (2);
  \node at (1.4,1.6) {\footnotesize  $\partial\Omega$};
  \node at (1.7,0) {\footnotesize $\Omega$};

  \node at (-0.9,0.9){\footnotesize  $\mathcal{C}_1$};
  \draw[draw=black][line width=0.2pt](-1,0.8) to [out=-30, in=-160](0,0.8)  (0,0.8) to [out=20, in=150](0.5,0.8) (0.5,0.8) to [out=-30, in=-150](1,0.8);
  \draw[draw=black, draw=red][line width=0.1pt](-1,0.805) to [out=-30, in=-160](0,0.805) (0,0.805) to [out=20, in=150](0.5,0.805) (0.5,0.805) to [out=-30, in=-150](1,0.805);
   \draw[draw=black][line width=0.2pt](-1,0.81) to [out=-30, in=-160](0,0.81)  (0,0.81) to [out=20, in=150](0.5,0.81) (0.5,0.81) to [out=-30, in=-150](1,0.81);
  \draw[draw=black][line width=0.2pt](-1,0.8)--(-1,0.81) (1,0.8)--(1,0.81);

  \draw[draw=white,fill=gray!50](-0.8,0.2)circle(0.2) (0.7,-0.3)circle(0.2);
  \node at(-0.8,0.2){$\cdot$};\node at  (-0.7, 0.2){\footnotesize $D_1$};
  \node at(0.7,-0.3){$\cdot$};\node at  (0.8, -0.3){\footnotesize $D_2$};

  \node at (-0.9,-0.8){\footnotesize $\mathcal{C}_k$};
   \draw[draw=black][line width=0.2pt](-1.2,-0.7) to [out=30, in=160](-0.7,-0.7)  (-0.7,-0.7) to [out=-20, in=200](0,-0.7) (0,-0.7) to [out=20, in=170](1.4,-0.7)  ;
  \draw[draw=black, draw=red][line width=0.1pt](-1.2,-0.71) to [out=30, in=160](-0.7,-0.71)  (-0.7,-0.71) to [out=-20, in=200](0,-0.71) (0,-0.71) to [out=20, in=170](1.4,-0.71)  ;
   \draw[draw=black][line width=0.2pt](-1.2,-0.72) to [out=30, in=160](-0.7,-0.72)  (-0.7,-0.72) to [out=-20, in=200](0,-0.72) (0,-0.72) to [out=20, in=170](1.4,-0.72)  ;
   \draw[draw=black][line width=0.2pt](-1.2,-0.7)--(-1.2,-0.72) (1.4,-0.7)--(1.4,-0.72);

   \node at (-1.5,0){\footnotesize  $\mathcal{C}_2$};
   \draw[draw=black][line width=0.2pt](-1.5,0.5)to[out=-100,in=70](-1.3,-0.4);
   \draw[draw=black, draw=red][line width=0.1pt](-1.505,0.5)to[out=-100,in=70](-1.305,-0.4);
   \draw[draw=black][line width=0.2pt](-1.51,0.5)to[out=-100,in=70](-1.31,-0.4);
   \draw[draw=black][line width=0.2pt](-1.5,0.5)--(-1.51,0.5) (-1.3,-0.4)--(-1.31,-0.4);
   \draw[densely dotted, draw=blue](-0.5,-0.9) rectangle(-0.2,-0.6);
   \draw[dashed,->] (-0.2,-0.9) -- (2.8, -1);
   \end{tikzpicture}
\begin{tikzpicture}[scale=6]
 \draw[densely dotted, draw=blue](-0.7,-1) rectangle(0,-0.3);
   \draw[draw=black][line width=0.2pt]  (-0.7,-0.5) to [out=-20, in=200](0,-0.5) ;
  \draw[draw=black, draw=red][line width=0.1pt]  (-0.7,-0.6) to [out=-20, in=200](0,-0.6) ;
   \draw[draw=black][line width=0.2pt] (-0.7,-0.7) to [out=-20, in=200](0,-0.7);

   \node at(-0.1,-0.62){\footnotesize $\mathcal{L}_k$};
   \draw[draw=blue,line width=0.3pt][->](-0.35,-0.67)--(-0.27,-0.67); \node at(-0.27,-0.7){\footnotesize ${ \tau}$};
   \draw[draw=blue!50!red,line width=0.3pt][->] (-0.35,-0.67)--(-0.35,-0.59);\node at(-0.30,-0.59){\footnotesize ${ \nu}$};
   \node at(-0.6,-0.735){$\cdot$};\node at(-0.55,-0.547){$\cdot$};\node at(-0.575,-0.64){$\cdot$};
   \node at(-0.575,-0.77){\footnotesize $x-\delta_k{ \nu}$};\node at(-0.56,-0.51){\footnotesize $x+\delta_k{ \nu}$};\node at(-0.585,-0.66){\footnotesize $x$};
   \draw[draw=black,line width=0.5pt, densely dotted][<->](-0.55,-0.547)--(-0.575,-0.64);
   \node at(-0.54,-0.6){\footnotesize $\delta_k$};
\end{tikzpicture}
\caption{ (left) Inside the domain $\Om$, there are highly insulating cracks $\mathcal C_k$ and highly conducting  $D_k$. (right) Crack $\mathcal C_k$ has uniform thickness of $\delta_k$. }\label{Fig:Cmodel}
\end{center}
\end{figure}
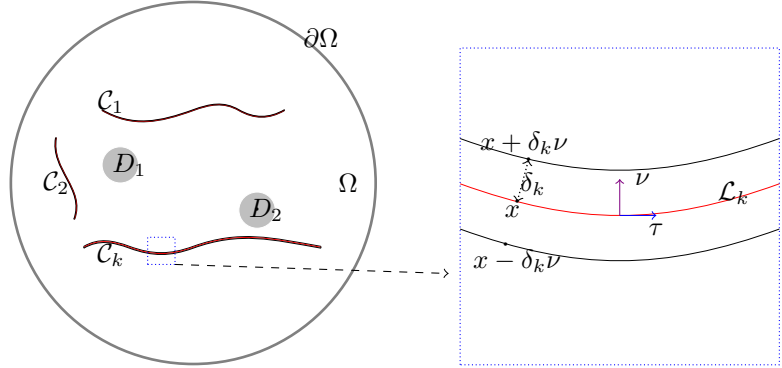

  We assume that $\mC_k$ and $D_k$ are well separated from each other as well as the boundary $\p\Omega$. To be precise, there exists a constant $d_0>0$ such that conditions:
\begin{equation}\label{separation}
\begin{array}{c}
  \inf_{k\neq k^\prime}\mbox{dist}(D_k,D_{k^\prime})\geq d_0, ~~\inf_{k\neq k^\prime}\mbox{dist}(\mathcal{C}_k,\mathcal{C}_{k^\prime})\geq d_0, \\
  \mbox{dist}(\mC,\p\Omega)\geq d_0, ~~\mbox{dist}(D,\p\Omega)\geq d_0,~~\mbox{dist}_{k,j}(\mC_k,D_j)\geq 2d_0.
\end{array}
\end{equation}
\section{Asymptotic expansions}
Since each crack $\mathcal C_k$ is highly insulating with very thin thickness, there is a noticeable potential jump along the crack \cite{Ammari2006,Kim2012} and the jump  changes with frequency. In this part we will focus on analyzing the frequency-dependent behaviors of the complex potential around the cracks. For a better understanding, we present the electrical current flux at different frequency ranges in Figure \ref{Fig:flux}. These figures clearly show how the electric current density changes with frequency in the presence of  both concrete cracks and reinforcing bars.
\subsection{Jump conditions}
To understand the phenomenon described in Figure \ref{Fig:flux}, we will start with investigating the jumps along sidewalls of concrete cracks by making use of Taylor expansion. By iterating the asymptotic formula for the crack, we can derive the leading-order term in the expansion of the boundary voltage when there are several well-separated cracks.
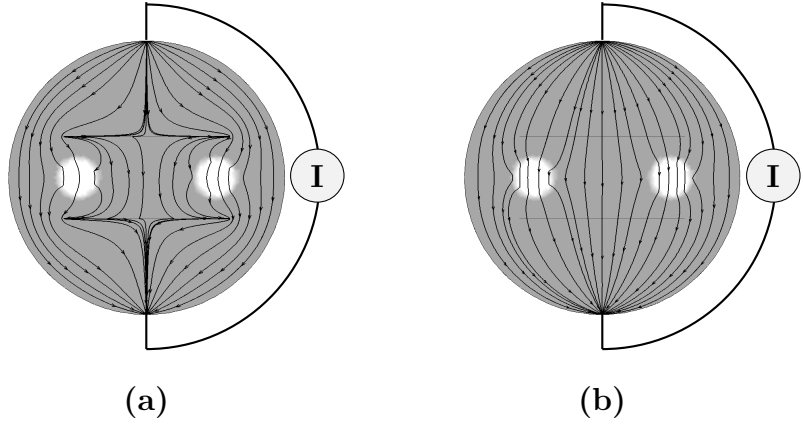
\begin{figure}[ht!]
  \begin{center}
    \begin{tikzpicture}
    \node at (-3,0) {\includegraphics[scale = 0.15]{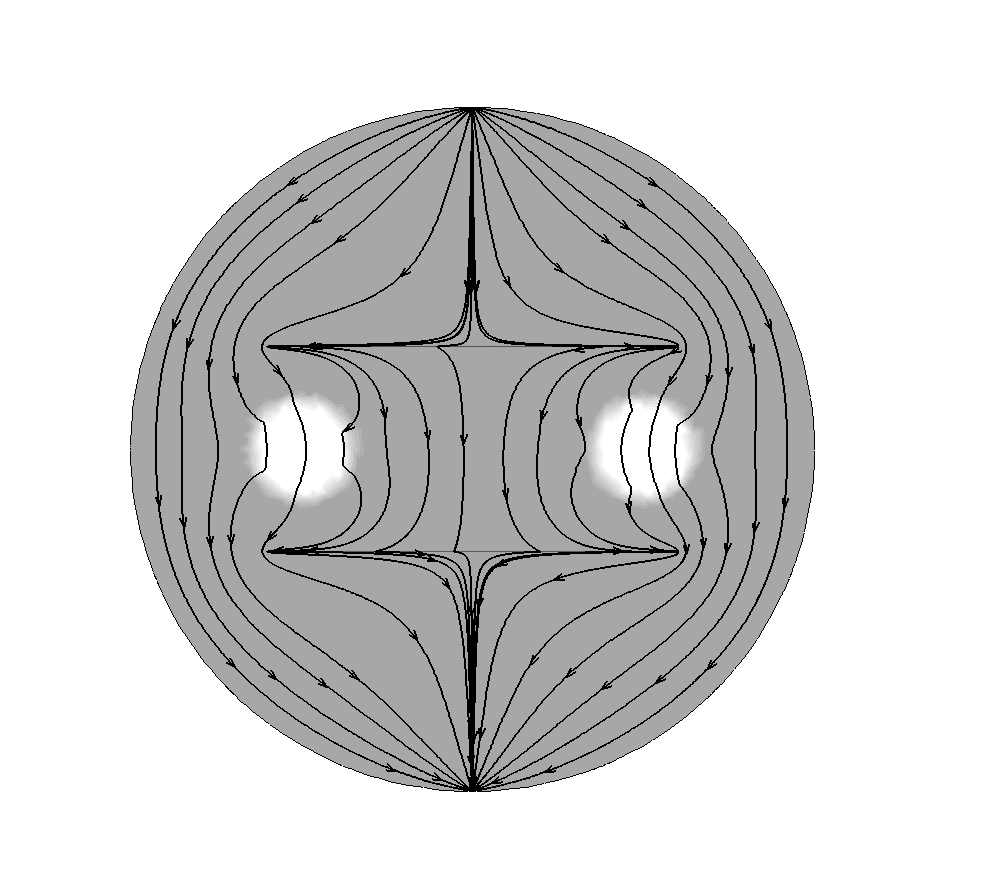}};
    \node at (3,0) {\includegraphics[scale = 0.15]{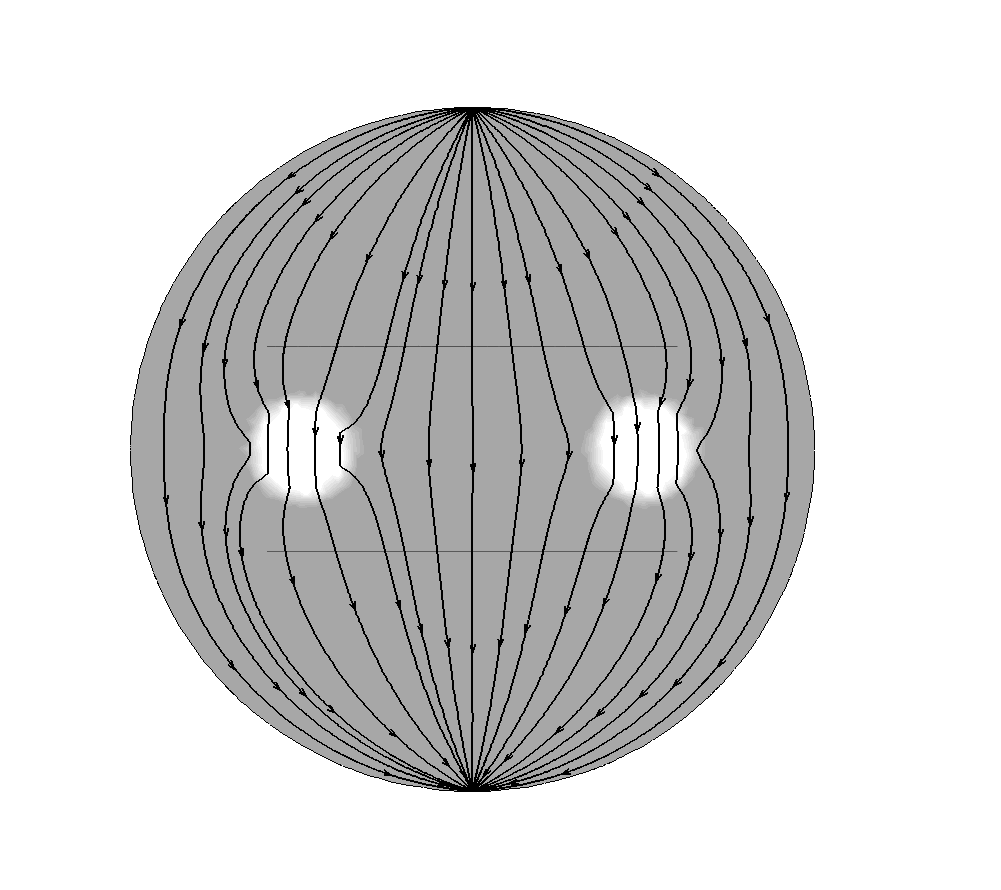}};
    \draw[thick] (-3.15,1.8)--(-3.15,2.3); \draw[thick] (-3.15,-2.3)--(-3.15,-1.8);
     \draw[thick] (2.85,1.8)--(2.85,2.3);  \draw[thick] (2.85,-2.3)--(2.85,-1.8);
   \draw [thick] (-3.15,-2.3) arc (-90:90:65pt);
   \draw[fill = gray!10] (-0.9,0) circle (10pt); \node at (-0.9,0)  {\textbf I};
     \draw [thick] (2.85,-2.3) arc (-90:90:65pt);
     \draw[fill = gray!10] (5.1,0) circle (10pt); \node at (5.1,0){\textbf I};
     \node at (-3.15, -3.0) {\textbf {(a)}};   \node at (2.85, -3.0) {\textbf {(b)}};
    \end{tikzpicture}
\caption{Electrical current flux for a concrete model with reinforcing bars (white) and cracks: (a) at low frequencies; (b) at  high frequencies.}\label{Fig:flux}
  \end{center}
  \end{figure}

For notational convenience, we define exterior($+$)/interior($-$)
normal derivative on the boundary of $\mathcal C_k$ as follows:
\[
\begin{array}{lll}
\f{\p u^\omega}{\p\nu}(x-\delta_k\nu_x )|_\pm &=& \lim_{s\rightarrow0^+}\f{\p u^\omega}{\p\nu}(x-\delta_k\nu_x \mp
s\nu_x ),\\
\f{\p u^\omega}{\p\nu}(x+\delta_k\nu_x )|_\pm&=
&\lim_{s\rightarrow0^+}\f{\p u^\omega}{\p\nu}(x+\delta_k\nu_x \pm s\nu_x )
\end{array} \quad~ \q(x\in \mathcal{L}_k),
\]
 where $\nu$ is the unit normal vector to the curve $\mathcal{L}_k$ as shown in Figure \ref{Fig:Cmodel}. Denote $[u^\omega]_k$ and $\left[\f{\p u^\omega}{\p\nu}\right]_k$ as jump of potential and jump of normal derivative across the boundary of crack $\mC_k$, respectively:
 \[
 \begin{array}{lll}
[u^\omega(x)]_k&:=&u^\omega(x+\delta_k\nu_x )-u^\omega(x-\delta_k\nu_x ) \\
 \left[ \f{\p u^\omega}{\p\nu}(x)\right]_k&:=& \f{\p u^\omega}{\p\nu}(x+\delta_k\nu_x )|_+-
 \f{\p u^\omega}{\p\nu}(x-\delta_k\nu_x )|_+
 \end{array}
 \quad~ \q(x\in \mathcal{L}_k).
 \]
 \begin{lemma}\label{Lemma:jump}
For $x\in \mathcal{L}_k$, the potential $u^\om$ and normal derivative $\f{\p u^\omega}{\p \nu}$ satisfy the following jump relations across the thin crack:
\begin{eqnarray}\label{Eq:VolJ}
[u^\omega(x)]_k = 2\delta_k\f{1}{\lambda_c(\omega)} \f{\p u^\omega}{\p\nu}(x-\delta_k\nu_x )|_+
+O\left((\delta_k)^2\right),
\end{eqnarray}
and
\begin{eqnarray}\label{Eq:NuJ}
 \left[ \f{\p u^\omega}{\p\nu}(x)\right]_k= - 2\delta_k\lambda_c(\omega) \f{\p^2 u^\omega}{\p\tau^2}(x-\delta_k\nu_x ) +  O \left((\delta_k)^2\right),
\end{eqnarray}
 where
\begin{equation}\label{Eq:lamC}
\lambda_c(\omega) =  \f{\sigma_c+i\omega\epsilon_c}{\sigma_b+i\omega\epsilon_b}.
\end{equation}
\end{lemma}
\begin{proof} Using the transmission condition of the potential $u^\om$ along the boundary of $\mathcal C_k$, we have for $x\in\mathcal L_k$
\begin{eqnarray*}
 \hspace{-0.5cm}u^\omega(x+\delta_k\nu_x )&=& u^\omega(x-\delta_k\nu_x )+2\delta_k\f{\p u^\omega}{\p\nu}(x-\delta_k\nu_x )|_- + O \left((\delta_k)^2\right),\\
 &=&u^\omega(x-\delta_k\nu_x )+2\delta_k\f{1}{\lambda_c(\omega)} \f{\p u^\omega}{\p\nu}(x-\delta_k\nu_x )|_+  + O \left((\delta_k)^2\right).
\end{eqnarray*}
Similarly, we use the transmission condition along the crack and Taylor expansion to get
\begin{eqnarray*}
\f{\p u^\omega}{\p\nu}(x+\delta_k\nu_x )|_+&=& \lambda_c(\omega)\f{\p u^\omega}{\p\nu}(x+\delta_k\nu_x )|_-\\
  &=&\lambda_c(\omega) \left(\f{\p u^\omega}{\p\nu}(x-\delta_k\nu_x )|_- +2\delta_k \f{\p^2 u^\omega}{\p\nu^2}(x-\delta_k\nu_x )|_- \right)+ O \left((\delta_k)^2\right),\\
  &=&\f{\p u^\omega}{\p\nu}(x-\delta_k\nu_x )|_+ - 2\lambda_c(\omega)\delta_k \f{\p^2 u^\omega}{\p\tau^2}(x-\delta_k\nu_x ) + O \left((\delta_k)^2\right),
\end{eqnarray*} with $\tau$ the unit tangent vector with respect to $x\in\mathcal{L}_k$.
\end{proof}

From the above asymptotic formulas (\ref{Eq:VolJ}) and (\ref{Eq:NuJ}), the jump of the potential and its normal derivative across the cracks depend on angular frequency $\omega$ as well as the thickness $\delta_k$. Therefore, the multi-frequency Cauchy data (multi-frequency current-voltage data $\{\Lambda_{\om_j}(g_k)~:~~k=1,\cdots, N_E,~j=1,\cdots, N_\omega\}$) reflects not only the geometry of cracks $\mathcal L_k$ but also its thickness $\delta_k$.  This is the major advantage of multi-frequency EIT system over the other existing non-destructive testing system.

\subsection{Effective zero-thickness crack model}
Based on Lemma \ref{Lemma:jump}, we can describe an effective zero-thickness crack model by imposing the jump conditions of $[u^\omega]_k$ and $\left[ \f{\p u^\omega}{\p\nu}\right]_k$ on the curves $\mathcal L_k$.  This means that  the potential $u^\om$ can be approximated by the corresponding potential $\widetilde u^{\om}$ satisfying the effective zero-thickness crack model \cite{Zribi2011,Poignard2007,Poignard2009}:
\begin{equation}\label{modelcell}
\begin{cases}
 &\nabla \cdot ((\gamma^{\omega}_b+ (\gamma^{\omega}_d-\gamma^{\omega}_b)\chi_D)\nabla \widetilde u^\om)= 0 ~~\q \textrm{in}\q \Om\setminus \cup_{k=1}^{N_C}\mathcal L_k,\\
&\left[\f{\p}{\p\nu}  \widetilde u^{\om}\right]_{\mathcal L_k} =0,~~~k=1,2,\ldots, N_C,\\
&\left[\widetilde u^\om\right]_{\mathcal L_k}= 2\delta_k\f{1}{\lambda_c(\omega)} \f{\p \widetilde u^\omega}{\p\nu}|_+,~~~k=1,2,\ldots, N_C,\\
& \gamma^\omega_b \f{\p \widetilde u^\omega}{\p \nu} = g \quad\mbox{on}~\p \Om,\\
\end{cases}\end{equation}
where $\chi_D$ is the characteristic function of $D$ and
 \begin{equation}\label{crack-model-jump}
  \begin{array}{lll}
[\widetilde u^\omega(x)]_{\mathcal L_k}&:=&\lim_{s\rightarrow0^+} \left(\widetilde u^\omega(x+s\nu_x )-\widetilde u^\omega(x-s\nu_x )\right) \\
 \left[ \f{\p \widetilde u^\omega}{\p\nu}(x)\right]_{\mathcal L_k}&:=& \lim_{s\rightarrow 0^+} \left(\f{\p \widetilde u^\omega}{\p\nu}(x+s\nu_x )- \f{\p \widetilde u^\omega}{\p\nu}(x-s\nu_x )\right)
 \end{array}
~\q(x\in \mathcal{L}_k).
\end{equation}

Since $u^\om \approx \tilde u^\omega$ in $\{ x\in \Om~: \mbox{dist}(x,\p\Om) < \f{d_0}{2}\}$, the forward model (\ref{Eq:uw}) and the effective zero-thickness crack model (\ref{modelcell}) have basically the same Neumann-to-Dirichlet data in terms of the inverse problem.
From now on, let $u^\om$ denote a solution of (\ref{modelcell}) for notational simplicity. From the above zero-thickness crack model, the boundary condition along curve $\mathcal{L}_k$ depends on thickness $\delta_k$ of concrete crack as well as the value of $\lambda_c(\omega)$ which is related with injected current frequency $\omega$. The aim of the following few sections is to derive an explicit formula for detecting positions of reinforcing bars and cracks by using asymptotic expansions of $u^\omega$. The explicit formula depends on injected frequency $\omega$ and crack thickness $\delta_k$. We consider separately the following two cases\cite{Zribi2011}:
\begin{itemize}
 \item High-frequency case: $\delta_k\approx 0$ and $0<c_0\leq|\lambda_c(\omega)|$.
 \item Low-frequency case: $|\lambda_c(\omega)|\approx 0$ and $\delta_k\approx 0$ with $|\lambda_c(\omega)|^{-1}\delta_k \approx \beta$ and $0<\beta<\infty$.
\end{itemize}
\subsection{High-frequency case: \boldmath{$\delta_k\approx 0$ and $0<c_0\leq|\lambda_c(\omega)|$ }} Let $u_0$ be the solution of equation (\ref{Eq:uw}) with $g= a\cdot\nu$ on $\p\Omega$ and $\gamma^\omega = 1$, where $a$ is a unit vector in $\mathbb{R}^2$.

Denote the fundamental solution of Laplace equation in two dimension as $\Gamma(x,x')$:
$$\Gamma(x,x'):= -\f{1}{2\pi} \ln|x-x'|,$$
and define the trace operator $\mathcal K_\Omega[\phi]$  for $\phi \in L^2(\p\Omega)$ by
\begin{equation}\label{Eq:komega}
  \mathcal{K}_\Omega [\phi](x) := \f{1}{2\pi}\int_{\p\Omega} \f{(x'-x)\cdot\nu(x')}{|x-x'|^2}\phi(x')ds_{x'},
  \quad x\in \p\Omega.
\end{equation}

In  high-frequency case, we suppose that the injected current frequency $\omega$ is not that low, so that $|\lambda_c(\omega)|$ is away from zero. When the thickness $\delta_k$ goes to zero, the potential jump along each crack $\mathcal{C}_k$ also goes to zero according to lemma (\ref{Lemma:jump}). Therefore the proposed problem can be regarded as traditional impedance boundary value problem and the influence of concrete crack on the  high-frequency current-voltage data is very weak as shown in Figure \ref{Fig:flux}(b).

In this case, the following boundary voltage asymptotic expansion  holds at high-frequencies. For detailed analysis and similar proof, one may refer to \cite{Ammari2006,Beretta2003,Beretta2001,Frieman1989}.
\begin{theorem}\label{Th:highasy}[Asymptotic expansion at high-frequencies]
For $x\in\p\Omega$, when the injection current frequency is high, the perturbations of voltage potential $u^\omega$ due to small inclusions $D_k$ and thin inclusions $\mC_k$ can be expressed as
\begin{eqnarray}\label{Eq:SepU}
&& \left(-\f{1}{2} {I}+\mathcal{K}_\Omega\right)[u^\omega-u_0](x)\nonumber \\
&=& - \sum_{k=1}^{N_C}\int_{\mathcal{L}_k} \delta_kA_k(x',\lambda_c(\omega))\nabla u_0(x')\cdot\nabla \Gamma(x,x')ds_{x'}\nonumber \\
&&-\delta_D^2 \sum_{k=1}^{N_D}\nabla \Gamma(x,z_k)\cdot
M(\lambda_d(\omega),B_k)\nabla u_0(z_k) + O (\delta_k^2)+ O
(\delta_D^3),
\end{eqnarray}
where $A_k(x,\lambda_c(\omega))$ is a $2\times 2$ symmetric matrix whose eigenvectors are
 $\nu_k(x)$ and $\tau_k(x)$ and the corresponding eigenvalues are $2(1-\f{1}{\lambda_c(\omega)})$ and $2\left(\lambda_c(\omega)-1\right)$, respectively. And $M(\lambda_d(\omega),B_k)$ is polarization tensor given by
\begin{equation}\label{Eq:GPT}
M_{ij}:=\int_{\p B_k}y^j(\lambda_d(\omega)I-\mathcal{K}^*_{B_k})^{-1}(\nu_x\cdot\na x^i)(y)ds_y,\quad i,j=1,2,
\end{equation}
with \begin{equation}\label{Eq:lamD}
  \lambda_d(\omega) = \f{(\sigma_d+\sigma_b)+i\omega(\epsilon_d+\epsilon_b)}{2((\sigma_d-\sigma_b)
  -i\omega(\epsilon_d-\epsilon_b))}. \end{equation}
\end{theorem}

Theorem \ref{Th:highasy} has  obvious meaning that the measured
boundary data is influenced by cracks and reinforcing bars since
the first term on right-side of formula (\ref{Eq:SepU}) only
related with cracks while the second term only related with
reinforcing bars. Depending on the magnitude of $\om, \delta_k$
and $\delta_D$, the dominative term on right-side of formula
(\ref{Eq:SepU}) may be alternative. To see the effect of $\om,
\delta_k$ and $\delta_D$ on the measured boundary data more
clearly, we need further analysis on the expansion formula in
Theorem \ref{Th:highasy}.

\begin{theorem}[Identification of cracks and bars]\label{Th:rebarC}
Let $\lambda_c(\omega)$ and $\delta_k$ satisfy the conditions stated in high-frequency case. Assume that all the cracks are line segments and all the bars are disks. Let $Q_k$ and $P_k$ denote the endpoints of the segment $\mathcal{L}_k$ and let $z_k$ denote the center of $D_k$.
Then $(-\f{1}{2}I+\mathcal{K}_\Omega)[u^\omega-u_0]$ on the boundary $\p\Om$ can be expressed as
\begin{align}\label{Eq:holo}
\Re\left\{ (-\f{1}{2}I+\mathcal{K}_\Omega)[u^\omega-u_0]({\bf x})\right\} = \Re\{G^\Re(\x)\}
+O (\delta_k^2)+ O (\delta_D^3),\\
\Im\{(-\f{1}{2} {I}+\mathcal{K}_\Omega)[u^\omega-u_0]({\bf
x})\}=\Re\left\{G^\Im(x)\right\} +O (\delta_k^2)+ O (\delta_D^3),
\end{align} where $G^\Re$ and $G^\Im$ are meromorphic  functions:
\begin{align}
\f{dG^\Re(\x)}{d\x}=\sum_{k=1}^{N_C}
{\mathfrak{C}_k^{\Re}(\om,\delta_k)}\left(\f{1}{\x-\Q_k}-\f{1}{\x-\P_k}\right)-
\sum_{k=1}^{N_D}{\mathfrak{D}_k^{\Re}(\om,\delta_D)}\f{1}{(\x-\z_k)^2}\label{Eq:holo2-1}\\
\f{d G^\Im(\x)}{d \x} = \sum_{k=1}^{N_C}{\mathfrak{C}_k^{\Im}(\om,\delta_k)}\left(\f{1}{\x-\Q_k}-\f{1}{\x-\P_k}\right)
-\sum_{k=1}^{N_D}{\mathfrak{D}_k^{\Im}(\om,\delta_D)}\f{1}{(\x-\z_k)^2}
 \label{Eq:holo2-2}\end{align}
and
 \begin{align}
\mathfrak{C}_k^{\Re}(\omega,\delta_k)&=\f{\delta_k}{\pi}\left(\Re\{(\lambda_c(\omega)-1)\}a_{\tau_k}+i\Re\{(1-\f{1}{\lambda_c(\omega)}) \} a_{\nu_k}\right)\label{Def:holo3-1}\\
\mathfrak{C}_k^{\Im}(\omega,\delta_k)&=\f{\delta_k}{\pi}\left(\Im\{\lambda_c(\omega)-1\}a_{\tau_k}+i\Im\{1-\f{1}{\lambda_c(\omega)}\} a_{\nu_k}\right)\label{Def:holo3-2}\\
{\mathfrak{D}}_k^{\Re}(\omega,\delta_D)&=-\Re\left\{\f{|B_k|\delta_D^2}{2\pi\lambda_d(\omega)}\right\}\a,
\quad \mathfrak{D}_k^{\Im}(\omega,\delta_D)
=-\Im\left\{\f{|B_k|\delta_D^2}{2\pi\lambda_d(\omega)}\right\}\a
.\label{Def:holo3-3}
 \end{align}
Here, $a_{\nu_k}=a\cdot\nu_k, a_{\tau_k}=a\cdot\tau_k,~
 {\x} =  x\cdot(1,i),~ {\bf a}=a\cdot(1,i), ~{\z_k} = z_k\cdot(1,i),~
  {\P_k}=P_k\cdot(1,i), ~\hbox{and} ~~{\Q_k}=Q_k\cdot(1,i).
$
\end{theorem}
\begin{proof}
Since $B_k$ is a disk, the formula (\ref{Eq:GPT}) gives
$M(\lambda_d(\omega),B_k)=\f{|B_k|}{\lambda_d(\omega)}I$.  Hence,
the formula (\ref{Eq:SepU}) in Theorem \ref{Th:highasy} can be
expressed as
\begin{equation}\label{Phi0}
 \left(-\f{1}{2} {I}+\mathcal{K}_\Omega\right)[u^\omega-u_0](x)=\Phi(x) + O (\delta_k^2)+
 O (\delta_D^3)\q \q (x\in\p\Om),
\end{equation}
where $\Phi$ is
\begin{equation}\label{Phi1}
\Phi(x)= -
\sum_{k=1}^{N_C}\delta_k\int_{\mathcal{L}_k}(A_k~a)\cdot\nabla
\Gamma(x,x')ds_{x'}-\f{\delta_D^2}{2\pi}
\sum_{k=1}^{N_D}\f{|B_k|}{\lambda_d(\omega)}\f{x-z_k}{|x-z_k|^2}\cdot
a .
\end{equation}
We use $a=a_{\nu_k}\nu_k+ a_{\tau_k}\tau_k$ to get
\begin{eqnarray}
&&\Phi(x)=-\f{1}{2\pi}\sum_{k=1}^{N_C}\delta_k\int_{\mathcal{L}_k}\left(2(1-\f{1}{\lambda_c(\omega)})
a_{\nu_k}\nu_k+ 2(\lambda_c(\omega)-1)a_{\tau_k}\tau_k\right)\cdot\f{x-x'}{|x-x'|^2}ds_{x'}\nonumber\\
&&\,\q
\q\q\q-\f{\delta_D^2}{2\pi\lambda_d(\omega)}\sum_{k=1}^{N_D}|B_k|\f{x-z_k}{|x-z_k|^2}\cdot
a \q\q\q\q(x\in\p\Om).
 \label{Phi2}
\end{eqnarray}

 From now on, we shall identify $\mathbb{R}^2$ with the complex plane $\mathbb{C}$. In order to avoid confusion, we will adopt the following notations: $x=(x_1,x_2)$  denotes a point in $\mathbb{R}^2$ and $\mathbf{x}=x_1+ix_2$ will be the corresponding point in ${\mathbb C}$. Similarly, $x'=(x_1',x_2')$, $z_k=(z_{k_1},z_{k_2})$, $a = (a_1,a_2)$ in $\mathbb{R}^2$ can be changed to ${\bf x}' = x_1'+i x_2'$, ${\z_k}=z_{k_1}+iz_{k_2}$ and $ {\bf a}= a_1+ia_2$ in $\mathbb{C}$. Since $\lambda_c(\omega)$, $\lambda_d(\om)$ as well as $u^\omega$ are complex, we will consider real and imaginary part of $\Phi(x)$ separately.

The real part of $\Phi(x)$ for $x\in\p\Om$  can be expressed as
\begin{eqnarray}\label{Eq:maincom}
\q\q \Re\{\Phi(\x)\}=\Re\left\{
-\f{1}{2\pi}\sum_{k=1}^{N_C}\delta_k\int_{\mathcal{L}_k}{
\f{\mbox{\boldmath $\xi$}}{\bf x-x'}}d s_{\bf
x'}-\Re\left\{\f{\delta_D^2}{2\pi\lambda_d(\omega)}
\right\}\sum_{k=1}^{N_D}|B_k| \f{\bf a}{\x-\z_k}\right\},
\end{eqnarray}
where \mbox{\boldmath $\xi$}$=\xi\cdot (1,i)$ and $\xi$ is
\begin{eqnarray}\label{xi}
\xi=\Re\{2(1-\f{1}{\lambda_c(\omega)})a_{\nu_k}\nu_k+ 2(\lambda_c(\omega)-1)a_{\tau_k}\tau_k\}.
\end{eqnarray}
Since  $\mathcal L_k$ is the segment with endpoints $P_k, Q_k$,  it can be written as $P_k+t(Q_k-P_k), 0\leq t\leq 1$. Therefore, $\mathcal L_k$ has its unit tangent vector ${\mbox{\boldmath $\tau$}_k} ={\f{\Q_k-\P_k}{|\P_k-\Q_k|}}$ and its unit normal vector ${\mbox{\boldmath $\nu$}_k} = i{ \f{\Q_k-\P_k}{|\P_k-\Q_k|}}$ in $\mathbb{C}$. Hence, the integral term in (\ref{Eq:maincom}) can be written as
\begin{eqnarray*}
  \int_{\mathcal{L}_k}{\bf \f{\mbox{\boldmath $\xi$}}{x-x'}}d s_{\bf x'} &=& {|\Q_k-\P_k|} \int_0^1 \f{\mbox{\boldmath $\xi$}}{{ (\x-\P_k)}-t{ (\Q_k-\P_k)}} dt\\
  &=& \f{ \mbox{\boldmath $\xi$}|\Q_k-\P_k|}{\Q_k-\P_k}\ln{\f{\x-\P_k}{\x-\Q_k}}.
\end{eqnarray*}
From (\ref{xi}), we have
\begin{eqnarray*}
 && \q\f{\mbox{\boldmath $\xi$}|\Q_k-\P_k|}{\Q_k-\P_k}\nonumber\\
 &&\q=\f{ |\Q_k-\P_k|}{\Q_k-\P_k}\left(2\Re\{1-\f{1}{\lambda_c(\omega)}\} a_{\nu_k}{\f{i(\Q_k-\P_k)}{|\P_k-\Q_k|}}+2\Re\{\lambda_c(\omega)-1\}a_{\tau_k}{\bf \f{Q_k-P_k}{|P_k-Q_k|}}\right)\nonumber\\
&&\q= 2\Re\{\lambda_c(\omega)-1\}a_{\tau_k}+i2\Re\{1-\f{1}{\lambda_c(\omega)}\} a_{\nu_k}.
\end{eqnarray*}
Therefore, (\ref{Eq:maincom}) can be simplified as
\begin{eqnarray}\label{Phi3}
\Re\{\Phi({\bf
x})\}=\Re\left\{\sum_{k=1}^{N_C}{\mathfrak{C}_k^{\Re}(\omega,\delta_k)}\ln{\f{\x-\Q_k}{\x-\P_k}}+
\sum_{k=1}^{N_D}{\mathfrak{D}_k^{\Re}(\omega,\delta_D)}\f{1}{\x-\z_k}\right\},
\end{eqnarray}
where $\mathfrak{C}_k^{\Re}(\omega,\delta_k)$ and $\mathfrak{D}_k^{\Re}(\omega,\delta_D)$ are the quantities defined in (\ref{Def:holo3-1}) and (\ref{Def:holo3-3}).
From (\ref{Phi3}), the real part of $\Phi$ can be viewed as the real part of the meromorphic function $G^{\Re}(\x)$ given by
\begin{eqnarray}\label{Gre}
G^{\Re}(\x):=\sum_{k=1}^{N_C}{\mathfrak{C}_k^{\Re}(\omega,\delta_k)}\ln{\f{\x-\Q_k}{\x-\P_k}}
+\sum_{k=1}^{N_D}{\mathfrak{D}_k^{\Re}(\omega,\delta_D)}\f{1}{\x-\z_k}.
\end{eqnarray}
Since $G^{\Re}(\x)$ is homomorphic except points $\P_k, \Q_k, \z_k$, it has complex derivative near $\p\Om$ in the complex plane:
\begin{eqnarray}\label{Gre2}
\f{d G^{\Re}(\x)}{d\x} =
\sum_{k=1}^{N_C}{\mathfrak{C}_k^{\Re}(\omega,
\delta_k)}\left(\f{1}{\x-\Q_k}-\f{1}{\x-\P_k}\right)-\sum_{k=1}^{N_D}{\mathfrak{D}_k^{\Re}
(\omega,\delta_D)}\f{1}{(\x-\z_k)^2}.
\end{eqnarray}
Similarly, we can give proof for the imaginary part of $\Phi(\x)$.
\end{proof}

The followings are remarks on Theorem \ref{Th:rebarC}:
\begin{remark}
According to Theorem \ref{Th:rebarC}, both $G^{\Re}(\x)$ and
$G^{\Im}(\x)$ can be viewed as known quantities  from the
knowledge of $(-\f{1}{2}I+\mathcal{K}_\Omega)[u^\omega-u_0]$ on
$\p\Om$. This theorem states that   $\f{d G^{\Re}({\bf x})}{d {\bf
x}}$ is a meromorphic function in $\mathbb{C}$ with simple poles
at the endpoints $\P_k,~\Q_k$ of the segments $\mathcal{L}_k$ and
poles of order $2$ at the center $\z_k$ of $D_k$.  Hence, the
residues of $\f{d G^{\Re}({\bf x})}{d {\bf x}}$ at the endpoints
are given by
\begin{eqnarray}\label{residue}
\mbox{Res} \left(\f{d G^{\Re}({\bf x})}{d {\bf x}},{\bf Q_k}\right)= \mathfrak{C}_k^{\Re}(\omega,\delta_k) = -
 \mbox{Res} \left(\f{d G^{\Re}({\bf x})}{d {\bf x}},{\bf P_k}\right).
\end{eqnarray}
The information of the center of $D_k$ is contained in the following function
\begin{eqnarray}\label{Eq:w}
w({\bf x})&:=&\sum_{k=1}^{N_D}{\mathfrak{D}_k^{\Re}}(\om,\delta_D)\f{1}{(\x-\z_k)^2}.
\end{eqnarray}
Then the function $\f{w'({\bf x})}{w({\bf x})}$ will have simple poles at poles of $w(\x)$. Hence, these center points can be identifies from boundary measurements \cite{Kang2004}.
\end{remark}
\begin{remark}
 To get some intuition of the frequency dependence of boundary data, let us look over the coefficients  $\mathfrak{C}_k^{\Re}(\omega,\delta_k)$ and $\mathfrak{D}_k^{\Re}(\omega,\delta_D)$ in the formula (\ref{Gre}). Recall that $\mathfrak{C}_k^{\Re}(\omega,\delta_k)$ is only related with concrete cracks $\mC_k$ while $\mathfrak{D}_k^{\Re}(\omega,\delta_D)$ is only related with reinforcing bars $D_k$. The $\lambda_c(\omega)$ in the quantity $\mathfrak{C}_k^{\Re}(\omega,\delta_k)$ satisfies
 $$C_1\f{\sigma_c+\om\epsilon_c}{\sigma_b+\om\epsilon_b}\leq|\lambda_c(\omega)|
 \leq C_2\f{\sigma_c+\om\epsilon_c}{\sigma_b+\om\epsilon_b},$$ where $C_1,C_2$ are positive constants independent of $\om$. Since $\lim_{\delta_k\rightarrow 0} |\mathfrak{C}_k^{\Re}(\omega,\delta_k)| =0$, the effect of the cracks at high frequencies is negligibly small. Therefore, the measured current-voltage data is mainly affected by reinforcing bars.

 As frequency decreases, $|\f{1}{\lambda_c(\omega)}|$ increases. Therefore, the quantity of term $\mathfrak{C}_k^{\Re}(\omega,\delta_k)$ becomes non-negligible. On the other hand, because $|\lambda_d(\omega)|$ does not change much with frequency, quantity of $\mathfrak{D}_k^{\Re}(\omega,\delta_D)$ varies little with respect to $\om$.
\end{remark}

The above analysis show that the effect of cracks on the boundary data highly depends on frequency, while the effect of reinforcing bars doesn't depend on frequency that much. This relation leads to the results that we can detect the reinforcing bars when frequency is  very high and  both cracks and reinforcing bars when frequency decreases.  Numerical simulations in the later part will show the verification of these analysis.
We can similarly analyze $\mathfrak{C}_k^{\Im}(\omega,\delta_k)$ and $\mathfrak{D}_k^{\Im}(\omega,\delta_D)$ as above remarks. For low frequency case, instead of applying the above theorem, we have the following results.
\subsection{Low-frequency case: {\boldmath $|\lambda_c(\omega)|\approx 0$ and $\delta_k\approx 0$ with $|\lambda_c(\omega)|^{-1}\delta_k \approx \beta$ and $0<\beta<\infty$}}

In low-frequency case,  the admittivity contrast $\lambda_c(\omega)$ is getting close to zero. As crack thickness $\delta_k$ goes to zero, we suppose that $\f{1}{\lambda_c(\omega)}\delta_k\approx\beta$. Then according to lemma \ref{Lemma:jump}, potential jump along each crack could not be ignored(see Figure \ref{Fig:flux}(a)). According to \cite{Ammari2006,Frieman1989,Zribi2011}, we have the following asymptotic expansion formula of the potential $u^{\om}$ for low frequency current.
\begin{theorem}[Asymptotic expansion at low frequencies]
In low-frequency case, we have the following asymptotic formula for the boundary perturbations of the potential $u^\omega$:
\begin{eqnarray}\label{Eq:lowAsympt}
 &\indent \displaystyle \left(-\f{1}{2}I+K_\Omega\right)[u^\omega-u_0](x) =-\delta_D^2 \sum_{k=1}^{N_D}\nabla \Gamma(x,z_k)\cdot M(\lambda_d,B_k)\nabla u_0(z_k)\nonumber\\
&  \indent \displaystyle +\sum_{k=1}^{N_C} \int_{\mathcal{L}_k}\f{\p \Gamma(x,x')}{\p \nu(x')} [u]_k(x')dx' +
O (\delta^2_k) +  O ((\delta_D)^3).
\end{eqnarray}
\end{theorem}
In this case, since the potential jump $[u^\omega]_k=2\delta_k\f{1}{\lambda_c(\omega)} \f{\p u^\omega}{\p\nu}(x-\delta_k\nu_x )|_+$ along $\mathcal L_k$ is very large and could not be ignored, the effect of reinforcing bars on the perturbations of the boundary voltage is hidden by cracks. Although we cannot write (\ref{Eq:lowAsympt}) in an explicit way, we know that it is related with the endpoints as well as the potential jump along the concrete cracks. When multiple concrete cracks are well separated from each other, we can always image them from boundary measurements. However, reinforcing bars at low frequencies are invisible since the concrete cracks will dominate the boundary measurements.
\subsection{Spectroscopic analysis}
Based on the above analysis in low-frequency case and high-frequency case, we mathematically derived the frequency dependency of the current-voltage data in a rigorously way. The current-voltage data is mainly affected by the outermost cracks when frequency is low, whereas the data mainly depends on the reinforcing bars when frequency is high. With this reason, we can detect the outermost cracks at low frequency. As frequency increases, the reinforcing bars become gradually visible whereas cracks fade out (thicker crack fades out at higher frequency than thinner crack). Hence, multi-frequency EIT system allows to probe these frequency dependent behavior.
\section{Numerical simulations}
We make use of three different numerical simulation models on a disk $\Omega=\{(x,y):x^2+y^2\leq (0.1)^2\}$ with radius unit $m$ as shown in figure \ref{Tb:Models}. We generate a finite element mesh of the disk using triangular elements.  Inside the disk, we place cracks and bars. Complex admittivity distribution for each model is chosen as shown in table \ref{Tb:admt}.
\begin{figure}[ht!]
\centering
\begin{tabular}{c|c|c}
\hline
Model 1& Model 2& Model 3\\
\hline\\
\includegraphics[scale=0.3]{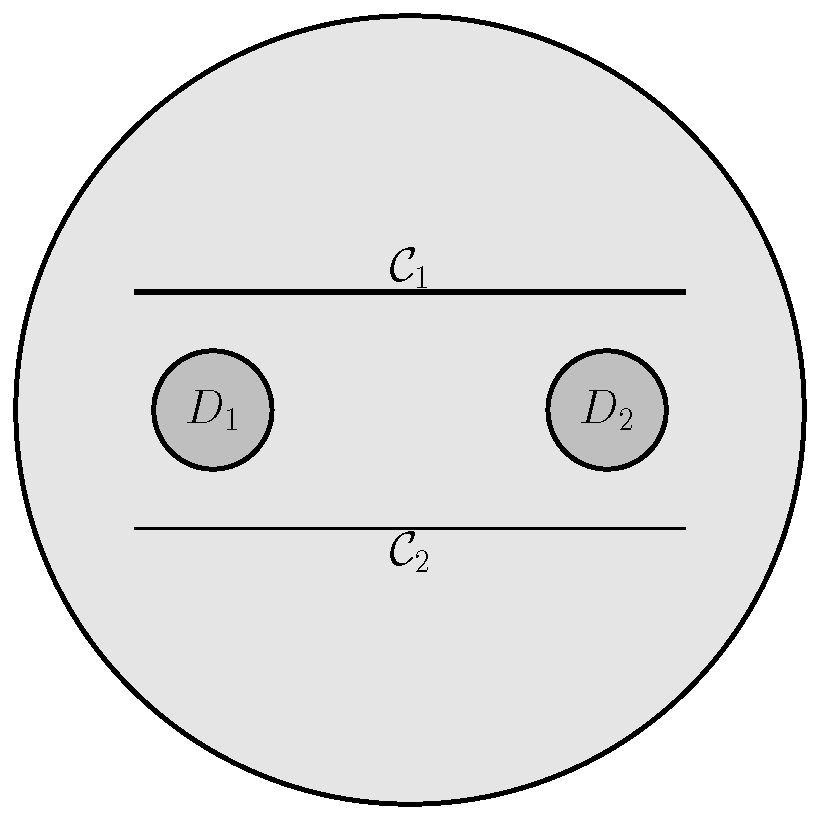}&
\includegraphics[scale=0.3]{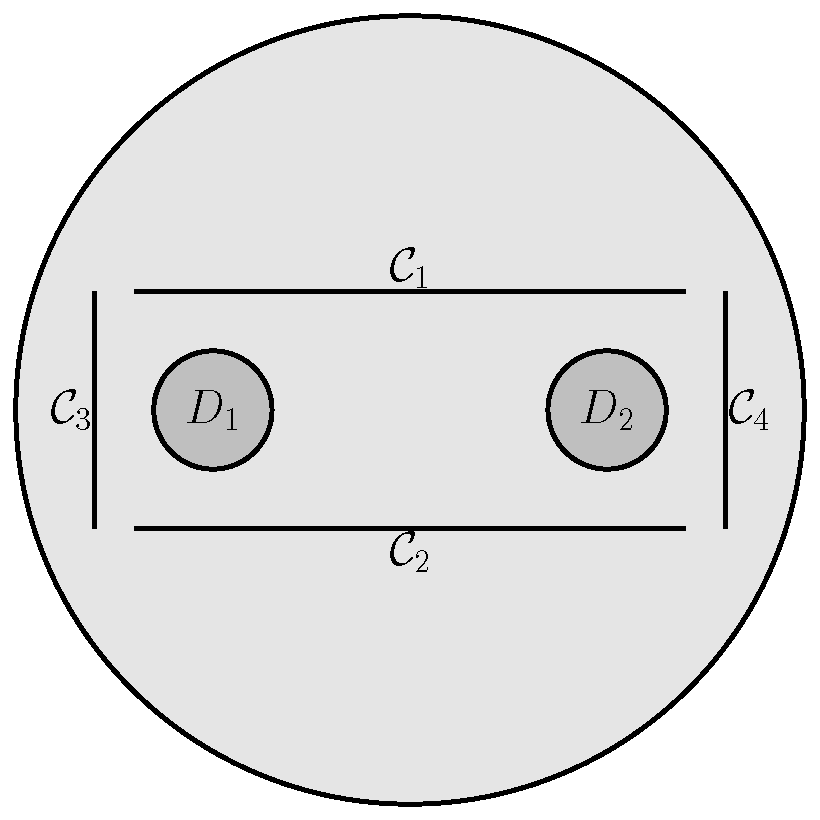}&
\includegraphics[scale=0.3]{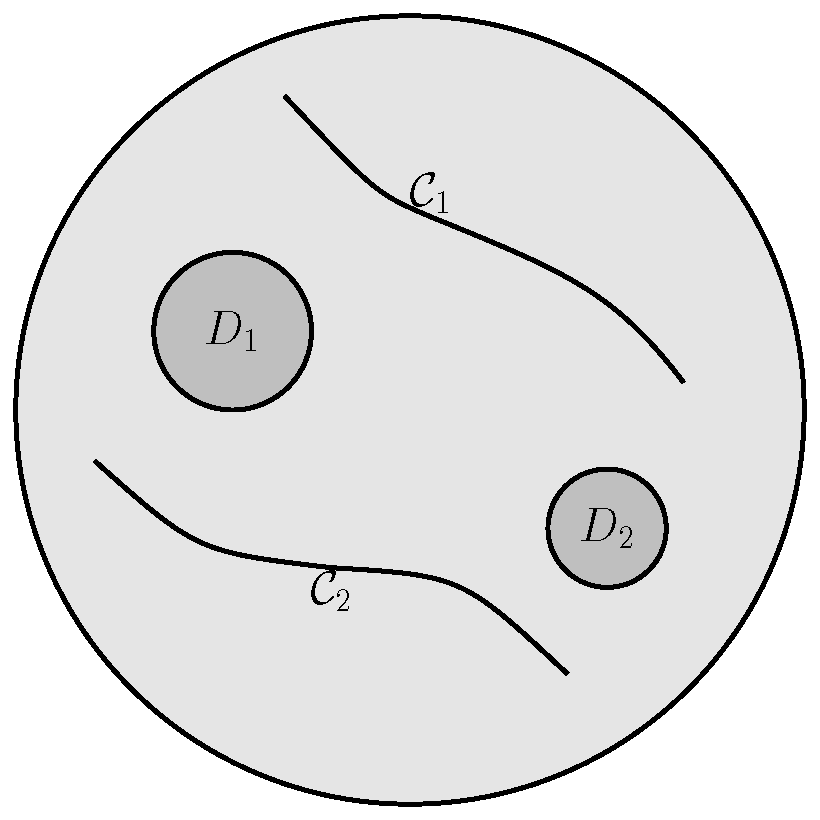}\\
\hline
\end{tabular}\caption{Models for numerical simulation.}
\label{Tb:Models}
\end{figure}
\begin{table}[h!]\centering
  \begin{tabular}{c|c}
    \hline
   Subdomain & Admittivity distribution\\ \hline
    $D_1,~D_2$ & $\q\q\gamma_d^\omega =10^5+i\omega*10^6*\epsilon_0\q\q$ \\
$\q\q\mC_1,~\mC_2,~\mC_3,~\mC_4\q\q$ & $\q\q\gamma_c^\omega =10^{-6}+i\omega*10^2*\epsilon_0\q\q$\\
Otherwise & $\gamma_b^\omega =1+i\omega*10^4*\epsilon_0$ \\ \hline
  \end{tabular}
  \caption{Admittivity distribution in each subdomain (\scriptsize $\epsilon_0 = 8.85*10^{-12}F/m$)}
  \label{Tb:admt}
\end{table}
In the numerical simulations, we use the standard 16-channel multi-frequency EIT system\cite{Kim2008, Oh2007,Seo2012} where 16 electrodes $\mE_1,\cdots,\mE_{16}$ are attached to $\p\Omega$ with uniform distance between two adjacent electrodes.
We inject 16  number of currents  using adjacent pairs of electrodes to generate simulated current-voltage data set
\begin{equation*}
\mathbb{F}_{\om}=\left[V_{\om}^{1,1},\cdots,V_{\om}^{1,16},
V_{\om}^{2,1},\cdots,V_{\om}^{2,16},\cdots\cdots,V_{\om}^{16,1},\cdots,V_{\om}^{16,16}\right]^T.
\end{equation*}
where $V_{\om}^{j,k}$ denotes the potential difference between electrodes $\mE_{k}$ and $\mE_{k+1}$ when the $j$-th current is injected using the adjacent  pair $\mE_{j}$ and $\mE_{j+1}$. To be precise, $V_{\om}^{j,k}=u^\om_j|_{\mE_k}-u^\om_j|_{\mE_{k+1}}$ is computed by solving the  following mixed boundary value problem
\begin{equation}\label{fw7}
 \left\{\begin{array}{l}
    \na \cdot\left( \gamma^\om \na u^\om_j\right)=0\quad \hbox{ in }\Omega\\
    1=-\int_{\mE_j} \gamma^\om\f{\p u^\om_j}{\p \nu} \;ds
    =\int_{\mE_{j+1}}\gamma^\om
    \f{\p u^\om_j}{\p \nu} ds\\
    \na u^\om_j\times \nu|_{\mE_k}=0,\q \int_{\mE_k} \gamma^\om\f{\p u^\om_j}{\p \nu} \;ds=0 ~~(k=1,\cdots, 16)
       \\
   \gamma_\om \f{\p u^\om_j}{\p \nu}=0 \q\mbox{ on } \p\Om\setminus
   \cup_{k=1}^{16}\mE_k,\q\q \int_{\p\Om} u^\om_j=0
    \end{array}
    \right.
\end{equation}
where the contact impedance is ignored for simplicity.
For  $\om=\om_1,\om_2,\cdots,\om_{N_\om}$ ranging from 10Hz to 1MHz, we get $N_\om$ data vectors  $\mathbb{F}_{\om_1},\mathbb{F}_{\om_2},\cdots, \mathbb{F}_{\om_{N_\omega}}$.
 The EIT reconstruction method makes use of the sensitivity matrix $\mathbb{S}$:
\begin{small}
\begin{equation*}
\mathbb{S}=
\left[ \begin{matrix}
S_1^{1,1}&S_2^{1,1}&\cdots&S_p^{1,1} &\cdots& S_{N_T-1}^{1,1}& S_{N_T}^{1,1}\\
 &&\cdots& \\
S_1^{1,16}&S_2^{1,16}&\cdots&S_p^{1,16} &\cdots&S_{N_T-1}^{1,16} & S_{N_T}^{1,16}\\
&&\vdots&\\
S_1^{16,1}&S_2^{16,1}&\cdots&S_p^{16,1} &\cdots& S_{N_T-1}^{16,1}&S_{N_T}^{16,1}\\
 &&\cdots& \\
S_1^{16,16}&S_2^{16,16}&\cdots&S_p^{16,16} &\cdots&S_{N_T-1}^{16,16}& S_{N_T}^{16,16}\\
\end{matrix}\right]_{(16)^2\times N_T},
\end{equation*}
\end{small}
where $N_T$ is the number of triangular elements and
$$
S_p^{j,k} = \int_{T_p}\na U_j(x)\cdot\na U_k(x)dx
$$ with $U_j$ being the solution of forward problem (\ref{fw7}) with $\gamma^\omega=1$ subject to $j-$th current injection between $\mE_j$ and $\mE_{j+1}$.

We reconstruct the spectroscopic conductivity and permittivity images by solving the following linear system:
$$\mathbb{S}\delta\gamma_{\om}=\mathbb{F}_{\om}-\mathbb{F}_{\om,0}$$
where $\mathbb{F}_{\om,0}$ is the collected current-voltage data in absence of anomaly. We will describe the numerical simulations case by case.

For  model 1 in figure \ref{Tb:Models}, there are two reinforcing bars $D_1=\{(x,y): (x+0.05)^2+y^2\leq (0.015)^2\}$, $D_2=\{(x,y): (x-0.05)^2+y^2\leq (0.015)^2\}$ and two thin concrete cracks $\mC_1=\{(x,y): |x|<0.07, |y-0.03|<5\times10^{-5}\}$, $\mC_2=\{(x,y): |x|<0.07, |y+0.03|<2.5\times10^{-5}\}$. The numerical simulation in figure \ref{Fig:recon1} shows that at low frequency only concrete cracks are visible; as frequency goes higher, thinner insulating crack begin to fade out whereas thicker crack is still visible; at high frequency, only reinforcing bars are visible.
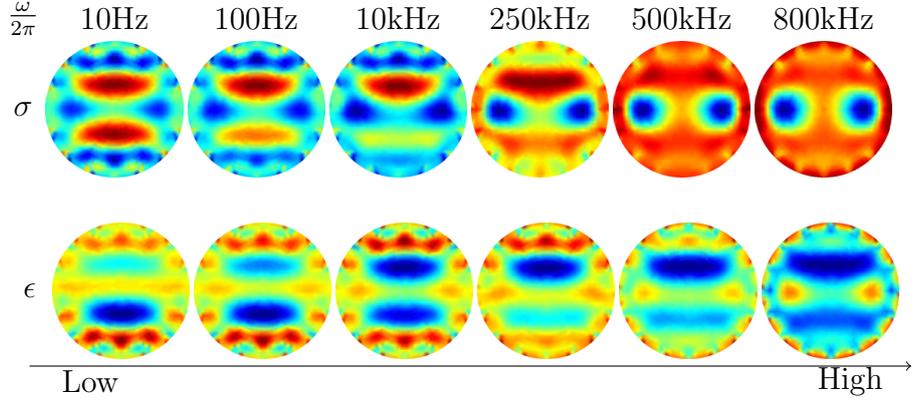
\begin{figure}[ht!]
  \centering
\begin{tikzpicture}[scale=0.8]
\node at (-8.5,0) {$\f{\omega}{2\pi}$};
\node at (-8.5,-1.5) {$\sigma$};
\node at (-7,0) {10Hz};
\node at (-4.6667,0) { 100Hz};
\node at (-2.3333,0) { 10kHz};
\node at (0,0) { 250kHz};
\node at (2.3333,0) { 500kHz};
\node at (4.6667,0) { 800kHz};
\node at (-7,-1.5) {\includegraphics[width=1.8cm]{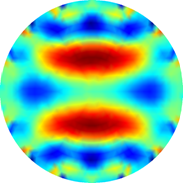}};
\node at (-4.6667,-1.5) {\includegraphics[width=1.8cm]{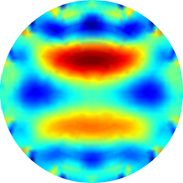}};
\node at (-2.3333,-1.5) {\includegraphics[width=1.8cm]{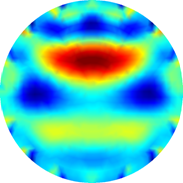}};
\node at (0,-1.5) {\includegraphics[width=1.8cm]{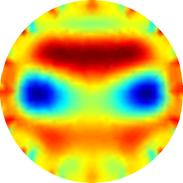}};
\node at (2.3333,-1.5) {\includegraphics[width=1.8cm]{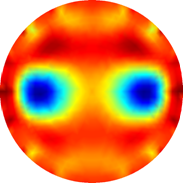}};
\node at (4.6667,-1.5) {\includegraphics[width=1.8cm]{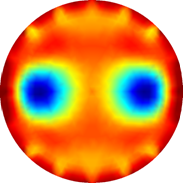}};
\end{tikzpicture}
\begin{tikzpicture}[scale=0.8]
\node at (-8.5,0) {$~~~~~~$};
\node at (-8.5,-1.5) {$\epsilon$};
\draw[->] (-8.5, -2.75) -- (6, -2.75);
\node at (-7,-1.5) {\includegraphics[width=1.8cm]{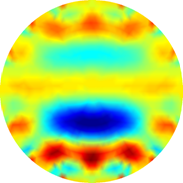}};
\node at (-4.6667,-1.5) {\includegraphics[width=1.8cm]{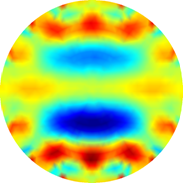}};
\node at (-2.3333,-1.5) {\includegraphics[width=1.8cm]{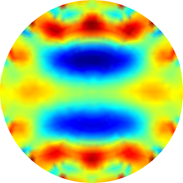}};
\node at (0,-1.5) {\includegraphics[width=1.8cm]{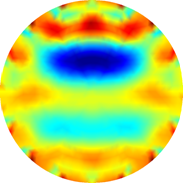}};
\node at (2.3333,-1.5) {\includegraphics[width=1.8cm]{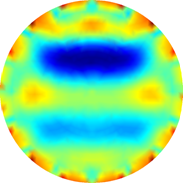}};
\node at (4.6667,-1.5) {\includegraphics[width=1.8cm]{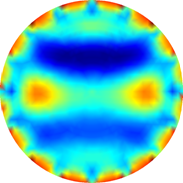}};
\node at (-7.5,-3) {Low}; \node at (5,-3) {High};
\end{tikzpicture}
\caption{Reconstructed admittivity image for model 1 using 16-channel multi-frequency EIT method: the first row is  $\sigma$ and the second row is  $\epsilon$.}
\label{Fig:recon1}
\end{figure}

For model 2 in figure \ref{Tb:Models}, two reinforcing bars $D_1=\{(x,y): (x+0.05)^2+y^2\leq (0.015)^2\}$, $D_2=\{(x,y): (x-0.05)^2+y^2\leq (0.015)^2\}$ are encircled by four concrete cracks; $\mC_1=\{(x,y): |x|<0.07, |y-0.03|<2.5\times10^{-5}\}$, $\mC_2=\{(x,y): |x|<0.07, |y+0.03|<2.5\times10^{-5}\}$,$\mC_3=\{(x,y): |x+0.08|<2.5\times10^{-5}, |y|<0.03\}$,$\mC_4=\{(x,y): |x-0.08|<2.5\times10^{-5}, |y|<0.03\}$. The simulation in figure \ref{Fig:recon2} shows that at low frequency, four outermost encircled cracks appear to be one object whereas reinforcing bars are hidden by cracks; as frequency increases, cracks gradually disappear whereas reinforcing bars begin to fade in.
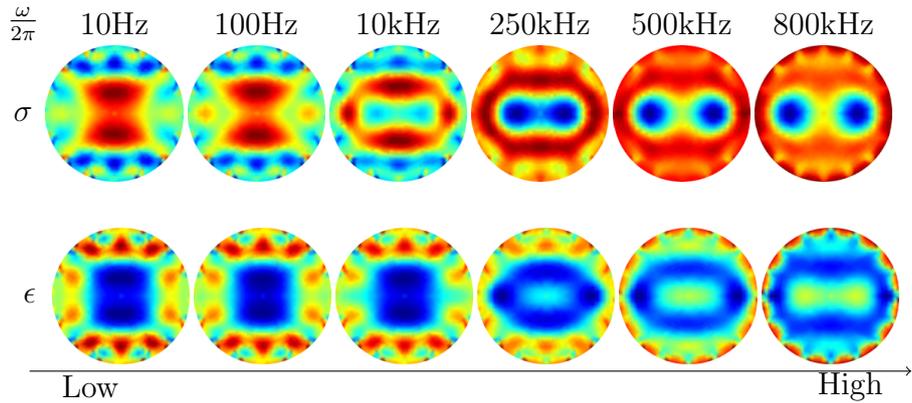
\begin{figure}[ht!]
  \centering
\begin{tikzpicture}[scale=0.8]
\node at (-8.5,0) {$\f{\omega}{2\pi}$};
\node at (-8.5,-1.5) {$\sigma$};
\node at (-7,0) {10Hz};
\node at (-4.6667,0) { 100Hz};
\node at (-2.3333,0) { 10kHz};
\node at (0,0) { 250kHz};
\node at (2.3333,0) { 500kHz};
\node at (4.6667,0) { 800kHz};
\node at (-7,-1.5) {\includegraphics[width=1.8cm]{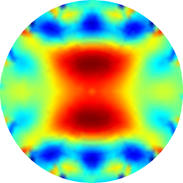}};
\node at (-4.6667,-1.5) {\includegraphics[width=1.8cm]{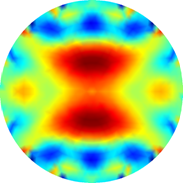}};
\node at (-2.3333,-1.5) {\includegraphics[width=1.8cm]{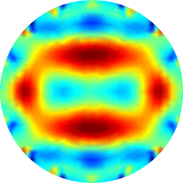}};
\node at (0,-1.5) {\includegraphics[width=1.8cm]{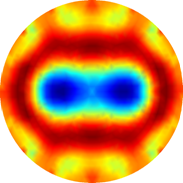}};
\node at (2.3333,-1.5) {\includegraphics[width=1.8cm]{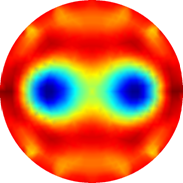}};
\node at (4.6667,-1.5) {\includegraphics[width=1.8cm]{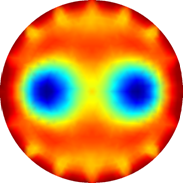}};
\end{tikzpicture}
\begin{tikzpicture}[scale=0.8]
\node at (-8.5,0) {$~~~~~~$ };
\node at (-8.5,-1.5) {$\epsilon$};
\draw[->] (-8.5, -2.75) -- (6, -2.75);
\node at (-7,-1.5) {\includegraphics[width=1.8cm]{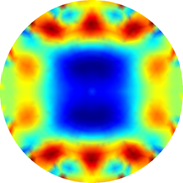}};
\node at (-4.6667,-1.5) {\includegraphics[width=1.8cm]{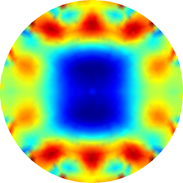}};
\node at (-2.3333,-1.5) {\includegraphics[width=1.8cm]{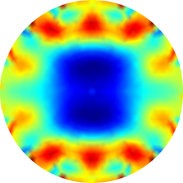}};
\node at (0,-1.5) {\includegraphics[width=1.8cm]{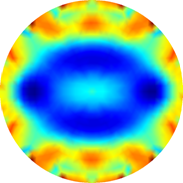}};
\node at (2.3333,-1.5) {\includegraphics[width=1.8cm]{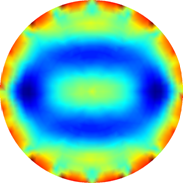}};
\node at (4.6667,-1.5) {\includegraphics[width=1.8cm]{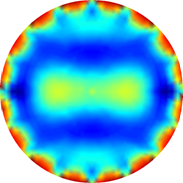}};
\node at (-7.5,-3) {Low}; \node at (5,-3) {High};
\end{tikzpicture}
\caption{Reconstructed admittivity image for model 2 using 16-channel multi-frequency EIT method: the first row is  $\sigma$ and the second row is  $\epsilon$.}
\label{Fig:recon2}
\end{figure}

For model 3 in figure \ref{Tb:Models}, there are two curved cracks with uniform thickness $5\times 10^{-5}m$ and two reinforcing bars $D_1=\{(x,y): (x+0.045)^2+(y-0.02)^2\leq (0.02)^2\}$, $D_2=\{(x,y): (x-0.05)^2+(y-0.03)^2\leq (0.015)^2\}$. The numerical simulations in Figure \ref{Fig:recon3} shows a similar behavior as in the previous simulations.
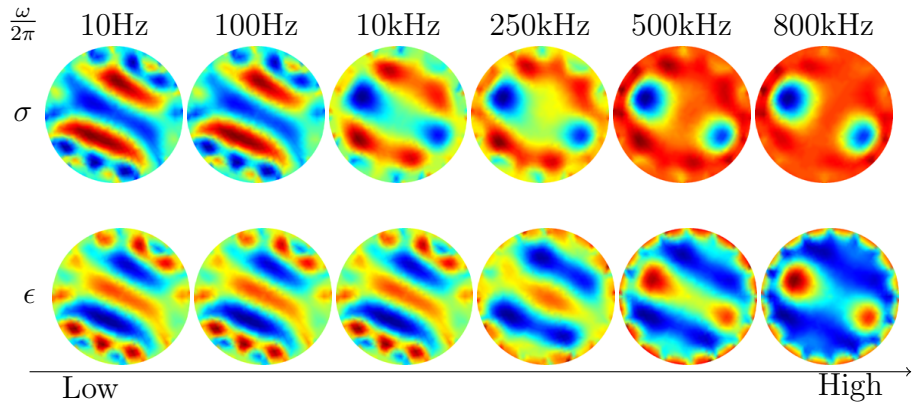
\begin{figure}[ht!]
  \centering
\begin{tikzpicture}[scale=0.8]
\node at (-8.5,0) {$\f{\omega}{2\pi}$};
\node at (-8.5,-1.5) {$\sigma$};
\node at (-7,0) {10Hz};
\node at (-4.6667,0) { 100Hz};
\node at (-2.3333,0) { 10kHz};
\node at (0,0) { 250kHz};
\node at (2.3333,0) { 500kHz};
\node at (4.6667,0) { 800kHz};
\node at (-7,-1.5) {\includegraphics[width=1.8cm]{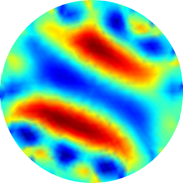}};
\node at (-4.6667,-1.5) {\includegraphics[width=1.8cm]{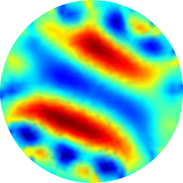}};
\node at (-2.3333,-1.5) {\includegraphics[width=1.8cm]{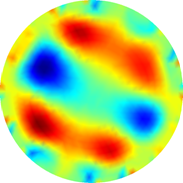}};
\node at (0,-1.5) {\includegraphics[width=1.8cm]{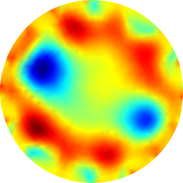}};
\node at (2.3333,-1.5) {\includegraphics[width=1.8cm]{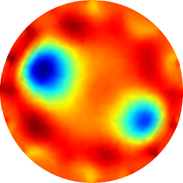}};
\node at (4.6667,-1.5) {\includegraphics[width=1.8cm]{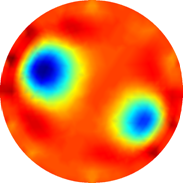}};
\end{tikzpicture}
\begin{tikzpicture}[scale=0.8]
\node at (-8.5,0) {$~~~~~~$ };
\node at (-8.5,-1.5) {$\epsilon$};
\draw[->] (-8.5, -2.75) -- (6, -2.75);
\node at (-7,-1.5) {\includegraphics[width=1.8cm]{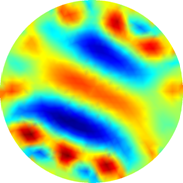}};
\node at (-4.6667,-1.5) {\includegraphics[width=1.8cm]{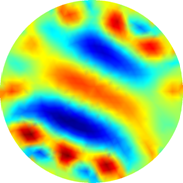}};
\node at (-2.3333,-1.5) {\includegraphics[width=1.8cm]{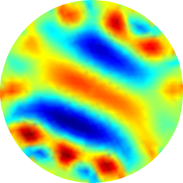}};
\node at (0,-1.5) {\includegraphics[width=1.8cm]{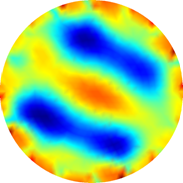}};
\node at (2.3333,-1.5) {\includegraphics[width=1.8cm]{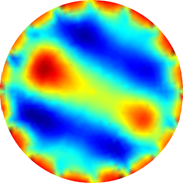}};
\node at (4.6667,-1.5) {\includegraphics[width=1.8cm]{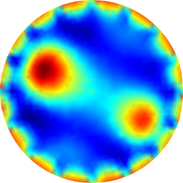}};
\node at (-7.5,-3) {Low}; \node at (5,-3) {High};
\end{tikzpicture}
\caption{Reconstructed admittivity image for model 3 using 16-channel multi-frequency EIT method: the first row is  $\sigma$ and the second row is  $\epsilon$.}
\label{Fig:recon3}
\end{figure}

\section{Conclusion}
In this work, we have developed two asymptotic expansions for
current-voltage data perturbations due to cracks and reinforcing
bars at various frequencies. Using these two asymptotic
expansions, we have mathematically shown that at high frequencies
we can visualize the reinforcing bars, while at low frequencies we
can only get the information of concrete cracks. Based on these
mathematical analysis, we conclude that multiple frequencies help
us to handle the spectroscopy behavior of the current-voltage data
with respect to cracks and reinforcing bars. When the frequency
increase from very low to very high, we can continuously observe
the images of cracks (low frequency), both cracks and reinforcing
bars (not too low, not too high frequency), and only reinforcing
bars (high frequency). The mathematical results are supported by
numerical illustrations.
\section*{Acknowledgements}

Ammari was supported  by the ERC Advanced Grant Project MULTIMOD--267184. Seo, Zhang and Zhou were supported by the National Research Foundation of Korea (NRF) grant funded by the Korean government (MEST) (No. 2011-0028868, 2012R1A2A1A03670512)

\bibliographystyle{siam}

\begin{thebibliography}{}

\end{thebibliography}


\begin{thebibliography}{1}

\bibitem{Ammari2006} {\sc H. Ammari, E. Beretta and E. Francini}, {Reconstruction of thin conductivity imperfections, II. the case of multiple segments}, Appl. Anal., 85(2006), pp. 87-105.

\bibitem{Ammari2004} {\sc H. Ammari and H. Kang}, {Reconstruction of small inhomogeneities from boundary measurements}, Lecture Notes in Math. 1846, Springer-Verlag, Berlin, 2004.

\bibitem{Ammari2007} {\sc H. Ammari and H. Kang}, {Polarization and moment tensors with applications to inverse problems and effective medium theory}, Appl. Math.Sci., Springer-Verlag, New York, 2007.

\bibitem{Ammari2009a} {\sc H. Ammari, H. Kang and E. Kim}, {Detection of internal corrosion}, ESAIM Proc., 26(2009), pp. 207-216.

\bibitem{Ammari2009b} {\sc H. Ammari, H. Kang, E. Kim, H. Lee and K. Louati}, {Vibration testing for detecting internal corrosion}, Stud. Appl. Math., 122(2009), pp. 85-104.

\bibitem{Ammari2011} {\sc H. Ammari, H. Kang, E. Kim, M. Lim and K. Louati}, {A direct algorithm for ultrasound imaging of internal corrosion}, SIAM J. Numer. Anal., 49(2011), pp. 1177-1193.

\bibitem{Ammari2008} {\sc H. Ammari, H. Kang, E. Kim, K. Louati and M. S. Vogelius}, {A music-type algorithm for detecting internal corrosion from electrostatic boundary measurements}, Numer. Math., 108(2008), pp. 501-528.

\bibitem{Beretta2003} {\sc E. Beretta, E. Francini and M. S. Vogelius}, {Asymptotic formulas for steady state voltage potentials in the presence of thin inhomogeneities. a rigorous error analysis}, J. Math. Pures Appl., 82(2003), pp. 1277-1301.

\bibitem{Beretta2001} {\sc E. Beretta, A. Mukherjee and M. S. Vogelius}, {Asymptotic formulas for steady state voltage potentials in the presence of conductivity imperfections of small area}, Z. Angew. Math. Phys., 52(2001), pp. 543-572.

\bibitem{Buettner1996} {\sc M. Buettner, A. Ramirez and W. Daily}, {Electrical resistance tomography for imaging concrete structures}, Structural materials technology an NDT conference, San Diego, C.A., U.S., 1996.

\bibitem{Beau2002} {\sc M. Chouteau and S. Beaulieu}, {An investigation on application of the electrical resistivity tomography method to concrete structures}, Geophysics, Los Angeles, CA, USA, 2002.

\bibitem{Colla1995} {\sc C. Colla, D. McCann, P. Das  and M. C. Forde}, {Investigation of a stone masonry bridge using electromagnetics, evaluation and strengthening of existing masonry structures}, RILEM, pp. 163-172, 1995.

\bibitem{Davis1989} {\sc J. L. Davis and A. P. Annan}, {Ground penetrating radar for high-resolution mapping of soil and rock stratigraphy}, Geophys. Prospect., 37(1989), pp. 531-551.

\bibitem{Diamond2006} {\sc G. G. Diamond, D. A. Hutchins, T. H. Gan, P. Purnell and K. K. Leong}, {Single-sided capacitive imaging for NDT}, Insight, 48(2006), pp. 724-730.

\bibitem{Frieman1989} {\sc A. Friedman and M. S. Vogelius}, {Identifcation of small inhomogeneities of extreme conductivity by boundary measurements: a theorem on continuous dependence},  Arch. Rational Mech. Anal., 105(1989), pp. 299-326.

\bibitem{Feldmann2008} {\sc C. Gerard and P. E. Feldmann}, {Non-destructive testing of reinforced concrete}, Structural Testing, (2008), pp. 13-17.

\bibitem{Hou2006} {\sc T. C. Hou and J. P. Lynch}, {Tomographic imaging of crack damage in cementitious structural components}, 4th International Conference on Earthquake Engineering, Taipei, Taiwan, 2006, pp. 1-10.

\bibitem{Guide} {\sc I.A.E.A.}, {Guidebook on non-destructive testing of
         concrete structures}, Training Course Serise No. 17, Int. Atomic Energy Agency, Austria, 2002.

\bibitem{Kang2004} {\sc H. Kang and H. Lee}, {Identifcation of simple poles via boundary measurements and an application of EIT},  Inverse Problems, 20(2004), pp. 1853-1863.

\bibitem{Kaipio2010a} {\sc K. Karhunen, A. Sepp{\"a}nen, A. Lehikoinen, J. Blunt, J. P. Kaipio and P. J. M. Monteiro}, {Electrical resistance tomography for assessment of cracks in concrete}, ACI Mat. J., 107(2010), pp. 523-531.

\bibitem{Kaipio2009} {\sc K. Karhunen, A. Sepp{\"a}nen, A. Lehikoinen, J. P. Kaipio  and P. J. M. Monteiro}, {Locating reinforcing bars in concrete with electrical resistance tomography}, Concrete Repair, Rehabilitation and Retrofitting II-Alexander et al. (eds), Taylor \& Francis Group, 2009.

\bibitem{Kaipio2010b} {\sc K. Karhunen, A. Sepp{\"a}nen, A. Lehikoinen, P. J. M. Monteiro  and J. P. Kaipio}, {Electrical resistance tomography imaging of concrete}, Cement and Concrete Research, 40(2010), pp. 137-145.

\bibitem{Zribi2011} {\sc A. Khelifi and H. Zribi}, {Asymptotic expansions for the voltage potentials with two-dimensional and three-dimensional thin interfaces}, Math. Meth. Appl. Sci., 34(2011), pp. 2274-2290.

\bibitem{Kim2012} {\sc S. Kim, E. J. Lee, E. J. Woo  and J. K. Seo}, {Asymptotic analysis of the membrane structure to sensitivity of frequency-difference electrical impedance tomography}, Inverse Problems, 28(2012).

\bibitem{Kim2008} {\sc S. Kim, J. Lee, J. K. Seo, E. J. Woo and H. Zribi}, {Multifrequency trans-admittance scanner: mathematical framework and feasibility}, SIAM J. Appl. Math., 69(2008), pp. 22-36.

\bibitem{McCann2001} {\sc D. M. McCann and M. C. Forde}, {Review of ndt methods in the assessment of concrete and masonry structures}, NDT\&E Int., 34(2001), pp. 71-84.

\bibitem{Oh2007} {\sc T. I. Oh, J. Lee, J. K. Seo, S. W. Kim and E. J. Woo}, {Feasibility of breast cancer lesion detection using a multi-frequency trans-admittance scanner (TAS) with 10 Hz to 500 kHz bandwidth}, Physiol. Meas., 28(2007), pp. S71-S84.

\bibitem{Poignard2007} {\sc C. Poignard}, {Asymptotics for steady-state voltage potentials in a bidimensional highly contrasted medium with thin layer}, Math. Meth. Appl. Sci., 31(2007), pp. 443-479.

\bibitem{Poignard2009} {\sc C. Poignard}, {About the transmembrane voltage potential of a biological cell in time-harmonic regime}, ESAIM Proc., 26(2009), pp. 162-179.

\bibitem{Sansalone1988} {\sc M. Sansalone and N. J. Carino}, {Impact-echo method: detecting honeycombing, the depth of surface-opening cracks, and ungrouted ducts}, Concrete Int.: Design and Cons., 10(1988), pp. 38-46.
\bibitem{Sansalone1986} {\sc M. Sansalone and N. J. Carino}, {Impact-echo: a method for flaw detection in concrete using transient stress waves}, Rep. No. NBSIR 86-3452, National Bureau of Standards, Gaithersburg, Md., 1986.

\bibitem{Seo2012} {\sc J. K. Seo and E. J. Woo.}, {Nonlinear inverse problems in imaging}, John Wiley \& Sons, 2012.

\bibitem{Stein2003} {\sc E. M. Stein and R. Shakarchi}, {Complex analysis}, Princeton University Press, 2003.

\bibitem{Yin2010} {\sc X. Yin, D. A. Hutchins, G. G. Diamond and P. Purnell}, {Non-destructive evaluation of concrete using a capacitive imaging technique: preliminary modelling and experiments}, Cement and Concrete Research, 40(2010), pp. 1734-1743.
\end{thebibliography}

\end{document}